\pdfoutput=1
\documentclass[12pt]{article}
\usepackage{lineno,hyperref,mathtools,amssymb,eqnarray,amsthm,enumerate,arydshln,chngpage,tikz,tikz-cd,fancyhdr}
\usepackage[short, nodayofweek]{datetime}
\makeatletter
\@addtoreset{chapter}{part}
\makeatother
\newtheorem{theorem}{Theorem}[section]
\newtheorem{lemma}[theorem]{Lemma} 
\newtheorem{proposition}[theorem]{Proposition} 
\newtheorem{corollary}[theorem]{Corollary} 
\theoremstyle{definition}
\newtheorem{definition}[theorem]{Definition}
\newtheorem{example}[theorem]{Example}
\newtheorem{notation}[theorem]{Notation}
\theoremstyle{remark}
\newtheorem{remark}[theorem]{Remark}

\title{
Some Braces of Cardinality $p^{4}$ and Related Hopf-Galois Extensions 
}

\author{
	D.  Pulji\'{c}\thanks{School of Mathematics, The University of Edinburgh, James Clerk Maxwell Building, The Kings Buildings, Mayfield Road EH9 3JZ, Edinburgh, \href{mailto:s1557452@sms.ed.ac.uk}{s1557452@sms.ed.ac.uk}}, A. Smoktunowicz\thanks{School of Mathematics, The University of Edinburgh, James Clerk Maxwell Building, The Kings Buildings, Mayfield Road EH9 3JZ, Edinburgh, \href{mailto: A.Smoktunowicz@ed.ac.uk}{A.Smoktunowicz@ed.ac.uk}}, K. Nejabati Zenouz\thanks{School of Computing and Mathematical Sciences, The University of Greenwich, Queen Mary's Building, Park Row, SE10 9LS, London, \href{mailto:K.NejabatiZenouz@greenwich.ac.uk}{K.NejabatiZenouz@greenwich.ac.uk}}
}

\date{\today}
\begin{document}
	\maketitle
	\begin{abstract}
	We describe all $\mathbb F_{p}$-braces of cardinality $p^{4}$ which are not right nilpotent. Our solution illustrates a general way of investigating  
	$\mathbb F_{p}$-braces of cardinality $p^{n}$ with a given multiplicative group. The constructed braces are left nilpotent, solvable and prime, and they also contain a non-zero strongly nilpotent ideal.   
	\end{abstract}
	\tableofcontents{}
	
	\section{Introduction}\label{S1}
	In 2005, W. Rump introduced braces as a generalisation of Jacobson radical rings in order to help study involutive set-theoretic solutions of the Yang-Baxter equation. Since then braces have been studied solely for their algebraic properties and have been linked to other research areas. For example, T. Gateva-Ivanova showed that braces are in correspondence with braided groups with an involutive braiding operator in \cite{Gateva}. Braces have also been shown to be equivalent to some objects in group theory such as bijective 1-cocycles and regular subgroups of the holomorph and they have been studied in connection to flat manifolds, quantum integrable systems, etc. Further, skew braces were introduced in \cite{VG}, as a generalisation of braces, and D. Bachiller followed to discover that there is a connection between skew braces and Hopf-Galois theory which was further studied in \cite{SVB}.
	
	The above evidences the value of classifying braces, which is a direction researches have taken in recent years. Braces with cyclic additive groups were classified in \cite{Rump1, Rump2} and braces of size $pq$ and $p^2q$ for primes $p$ and $q$ were classified in \cite{pq} and \cite{Diet}, respectively. All braces of cardinality $p^3$ have been described by D. Bachiller \cite{Ba} and all skew braces of the same size by K. Nejabati Zenouz \cite{kayvan2}. In  \cite{V} L. Vendramin suggested that the problem of classification of braces of order $p^n$ could be tackled by considering $\mathbb{F}$-braces, which is the framework of this paper.
	
	The general aim of our paper is to develop ring theoretic themes in the
	main calculations of characterising the braces as opposed to more group theoretic themes. We also partially answer a question asked in \cite{V}. Our main result  is Theorem \ref{24}, which describes all $\mathbb F_{p}$-braces of cardinality $p^{4}$ which  are  not right nilpotent. We also investigate connections of our results with Hopf-Galois extensions.
	
	\section{Background Information}
	Recall that a set $A$ with binary operations $+$ and $*$ is a {\em  left brace} if $(A, +)$ is an abelian group and the following version of distributivity combined with associativity holds:
	\begin{align}\label{BR}
		(a+b+a*b)* c&=a* c+b* c+a* (b* c), \\  a* (b+c)&=a* b+a* c \nonumber
	\end{align}
	for all $a, b, c\in A$.  Moreover, $(A, \circ )$ is a group, where we define \[a\circ b=a+b+a* b.\]
	
	In what follows, we will use the definition in terms of operation `$\circ $' presented in \cite{cjo} (see \cite{rump} for the original definition): a set $A$ with binary operations of addition $+$, and multiplication $\circ $ is a brace if $(A, +)$ is an abelian group, $(A, \circ )$ is a group and for every $a,b,c\in A$
	\[a\circ (b+c)+a=a\circ b+a\circ c.\] 
	Circle algebras related to braces were introduced by Catino and Rizzo in \cite{Catino}. We now recall Definition $2$ from \cite{Rump}, which we state for left braces, as it was originally stated for right braces. Notice that 
	$\mathbb F$-braces are related to circle algebras.
	
	\begin{definition} 
		Let $\mathbb F$ be a field. We say that a left  brace $A$ is an  $\mathbb F$-brace if its additive group is an $\mathbb F$-vector space such that $a*({\alpha }b)={\alpha }(a* b)$ for all $a,b\in A,$
		$ {\alpha }\in \mathbb F$. Here $ a*b=a \circ b -a -b$.
	\end{definition}
	In \cite{rump} Rump introduced {\em left nilpotent}  and  {\em right nilpotent}  braces and radical chains $A^{i+1}=A*A^{i}$ and $A^{(i+1)}=A^{(i)}*A$  for a left brace $A$, where  $A=A^{1}=A^{(1)}$. Recall that a left brace $A$ is left nilpotent if there is a number $n$ such that $A^{n}=0$, where inductively $A^{i}$ consists of sums of elements $a*b$ with $a\in A, b\in A^{i-1}$. A left brace $A$ is right nilpotent if there is a number $n$ such that $A^{(n)}=0$, where $A^{(i)}$ consists of sums of elements $a*b$ with $a\in A^{(i-1)}, b\in A$. Strongly nilpotent braces and the chain of ideals $A^{[i]}$ of a brace $A$ were defined in \cite{Engel}. Define $A^{[1]}=A$ and $A^{[i+1]}=\sum_{j=1}^{i}A^{[j]}*A^{[i+1-j]}$. A left brace $A$ is {\em strongly nilpotent} if there is a number $n$ such that $A^{[n]}=0$, where $A^{[i]}$ consists of sums of elements $a*b$ with $a\in A^{[j]}, b\in A^{[i-j]}$ for all $0<j<i$. 
	
	All braces and $\mathbb F_{p}$-braces considered in this paper are left braces.
	
	\section{Groups and Braces}
	In  Example $2$ in \cite{rump} Rump gave an example of a right brace which is right but not left nilpotent. By taking the opposite multiplication in this brace we obtain a left brace of cardinality $2^{4}$ which is left but not right nilpotent. Moreover this brace is an $\mathbb F_{2}$-brace as it is a vector space over $\mathbb F_{2}$. In this section we show that, for a prime number $p>3$,  multiplicative groups of left $\mathbb F_{p}$-braces of cardinality $p^{4}$ which are not right nilpotent are isomorphic to either group  XIV or group  XV described below.
	
	We first recall a result of David Bachiller from \cite{DB}.
	\begin{theorem} [Lemma 2.5.1, \cite{DB}]\label{1}
		Let $p$ be a prime number, and let $(B, +, \circ )$ be a finite left brace with $(B,+)$ isomorphic to $C_{p}^{m}$. Assume $m + 1 \leq  p$. Then, the order of each non-trivial element in the multiplicative group $(B, \circ )$ is equal $p$.
	\end{theorem}
	Therefore, every element in the multiplicative group of any $\mathbb F_{p}$-brace of cardinality $p^{4}$ for $p>3$ has order $p$. In \cite{Burnside} Burnside described the list of all braces of cardinality $p^{4}$ for prime  $p>3$. 
	
	Notice that if a brace  $A$ of cardinality $p^{4}$ has an abelian multiplicative group $(A, \circ )$, then it is a ring, and hence it is both left and right nilpotent \cite{cjo}. Therefore, a left nilpotent $\mathbb F_{p}$-brace which is left but not right nilpotent has a non-abelian multiplicative group. Therefore, using the list from pages 100, 101  from \cite{Burnside} we have the following possibilities of non-abelian  groups of cardinality $p^{4}$ in which each element has order $p$.
	\begin{itemize}
		\item (Group XIV) Group  (with multiplication $\circ $) generated by elements $P, Q, R, S$ which is the direct product of the group generated by $Q$ and the group generated by elements  $P,  R, S$. Moreover, $P\circ R=R \circ P$, $P\circ S=S \circ P$, $Q\circ P=P\circ Q$ and $R\circ S=S\circ R\circ P$.
		
		\item (Group  XV) Group (with multiplication $\circ $) generated by elements $P, Q$, $R, S$ with $P$ being in the centre of this group, and the centre of this group consists of elements $P^{i}$. Moreover, $R\circ Q=Q\circ R$ and $Q\circ S=S\circ Q\circ P$ and $R\circ S=S\circ R\circ Q$. 
	\end{itemize}
	Observe that group XIV has centre of cardinality $p^{2}$ and group $XV$ has centre of cardinality $p$, where by centre we mean the set of elements which commute under operation $\circ $ with all elements in $A$.
	
	We shall let $p$ be a prime number larger than $3$. The aim of this paper is to characterise all not right nilpotent braces of cardinality $p^{4}$. We will proceed in the following way which may be also applied in other situations.
	\begin{itemize}
		\item Let $G$ be a group of order $p^{n}$ and $g_{1}, g_{2}, \ldots , g_{n}$ be its generators. Let $f_{1}, f_{2}, \ldots , f_{n}$ be the defining relations of our group $G$.
		\item We will characterise $\mathbb F_{p}$-braces $A$ for which $A=\mathbb F_{p}g_{1}+\mathbb F_{p}g_{2}+\ldots +\mathbb F_{n}g_{n}$.
		\item We write $g_{i}*g_{j}=\sum_{0<k\leq n}\alpha _{i,j,k}g_{k}$ where $\alpha _{i,j,k}\in \mathbb F_{p}$ are unknown. As usual $a*b=a\circ b-a-b$ in our brace.  
		\item We substitute  these equations in the relations $f_{1}, f_{2}, \ldots , f_{n}$ as well as in the relations  $f_{i}*g_{j}$ for $i,j\leq n$ and find $\alpha _{i,j,k}$ such that the above equations hold.
		This gives necessary conditions for our brace to be well-defined.
		\item To check if the brace is well-defined we proceed as follows: We consider the linear space $B=\mathbb F_{p}g'_{1}+\ldots \mathbb F_{p}g'_{n}$ for some elements $g_{1}', g_{2}', \ldots , g_{n}'$ from some set $X$. We define maps  \[\lambda _{g_{i}}(g'_{j})=g'_{j}+\sum_{0<k\leq n}\alpha _{i,j,k}g'_{k}.\]
		\item  Next we define  maps $\lambda _{g}$ for $g\in G$ inductively by using the formula $\lambda _{g\circ h}(b)=\lambda _{g}(\lambda _{h}(b))$ for $b\in B$. To construct a well defined brace we then use Theorem 2.1 from \cite{Rio} with the cocycle $f(g)=\lambda _{g}(1)-1$. Notice that to be able to apply Theorem $2.1$ we need to show that $f:G\rightarrow B$ is a bijective cocycle. In our case it is done in section \ref{BC}.
		\item In this way we  can construct all $\mathbb F_{p}$-braces with multiplicative group $G$ and generators $g_{1}', g_{2}', \ldots , g_{n}'$ and such that $G=\mathbb F_{p}g_{1}'+\ldots +\mathbb F_{p}g_{n}'$ (see Section \ref{allbraces} for a method).
		\item For braces with multiplicative group XV  we can show that we can  assume that $A=\mathbb F_{p}g_{1}+\ldots +\mathbb F_{n}g_{n}$, provided that the brace is not right nilpotent. This then implies that we constructed all not right nilpotent braces with the multiplicative group $G$.
		\item We also show that all $\mathbb F_{p}$-braces with multiplicative group XIV are right nilpotent, and hence strongly nilpotent. Braces of cardinality $p^{4}$ for $p=2,3$ can be calculated using the GAP package.
	\end{itemize}
	
	\section{Braces whose Multiplicative Group is XV}
	Let $(G, \circ )$ be the group XV, so the free group generated by elements $P, Q, R, S$ subject to relations 
	\[Q\circ S=S\circ Q\circ  P,\]
	\[Q\circ R=R\circ Q,\]
	\[P\circ S=S\circ P, P\circ Q=Q\circ P, P\circ R=R\circ P,\]
	\[R\circ S=S\circ R\circ Q.\]
	\[P^{p}=Q^{p}=R^{p}=S^{p}=1_{G}.\]
	Let $(A, +, \circ)$ be a brace whose multiplicative group is $G$. As usual we have $g\circ h=g*h+g+h$. 
	
	We assume that $A$ is not right nilpotent. We recall a result of Rump \cite{rump}.
	\begin{theorem}[\cite{rump}]
		If $A$ is a brace of a cardinality $p^{n}$ for some prime number $p$ and some natural $n$, then $A^{n+1}=0$ where $A^{1}=A$ and $A^{i+1}=A*A^{i}$. 
	\end{theorem}
	We get the following Corollary.
	\begin{corollary}\label{5}
		Let $A$ be a brace of cardinality $p^{4}$. If $A^{4}\neq 0$, then $A^{i}/A^{i+1}$ has cardinality $p$ for $i=1,...,4$. 
	\end{corollary}
	\begin{proof} 
		Notice that if $A^{i}/A^{i+1}$ has cardinality smaller than $p$, then $A^{i}=A^{i+1}$, since $A^{i}$ is an abelian group under the operation $+$. Suppose that $A^{i}=A^{i+1}$ for some $i$. Then $A^{i}=A^{i+1}=A*A^{i+1}$ for some $i<4$. Then $A^{i}=A*A^{i}=A*{A^{i+1}}=A^{i+2}$. Continuing in this way we get that $A^{i}=A^{4}=A^{5}=0$, which is a contradiction, since we assumed that $A^{n}\neq 0$. Now we find a chain of subgroups of index at least $ p $ given by $ A^{4} \subset A^{3} \subset A^{2} \subset A $ which implies that $A^{i}/A^{i+1}$ has cardinality $p$ since the size of $ A $ is $ p^{4} $.
	\end{proof}

	\begin{notation}
		Let  $(A, +, \circ )$ be a brace. For an element $a\in A$ we will denote $a^{-1}$  an element in $A$ such that $a\circ a^{-1}=1=0$. Similarly, by $a^{n}$, sometimes written at $ a^{n\circ} $, we will denote the product of $n$ elements $a$ under the operation $\circ $.
	\end{notation}
	Note that in any brace the identity element $0$ of the additive group  $(A,+)$ coincides with the identity element $1$ of the multiplicative group $(A, \circ)$.

	\subsection{Some Supporting Lemmas}
	Recall that the $\lambda $ map in a brace $(A, +, \circ)$ is defined for $a,b\in A$ by \[ \lambda_{a}(b)=a*b+b=a\circ b-a.  \]
	It is known that \[\lambda _{a\circ b}(c)=\lambda _{a}\left(\lambda _{b}(c)\right).\]

	\begin{lemma}\label{2} 
		Let $A$ be an $ \mathbb{F}_{p} $-brace of cardinality $p^{4}$. Then for $a,b,c\in A$ we have 
		\begin{align*}
			\left(a^{-1}\circ b^{-1}\circ a\circ b\right)*c&=a*\left(b*c\right)-b*\left(a*c\right)\\
			&+a*\left(b*\left(a*c\right)\right)+b*\left(b*\left(a*c\right)\right)\\
			&-b*\left(a*\left(b*c\right)\right)-a*\left(a*\left(b*c\right)\right).
		\end{align*}
	\end{lemma}
	\begin{proof} 
		Note we have $A^{5}=0$. Now using properties of $ \lambda $ and the formula
		\[\lambda^{-1}_{a}(b)= b-a\ast b+ a\ast \left(a\ast b\right) -a\ast\left(a\ast \left(a\ast b\right)\right)\ \text{for}\ a,b\in A \]
		the result follows by computation.
	\end{proof}
	Recall that in group XV we have the relation \[Q=R^{-1}\circ S^{-1}\circ R\circ S.\]
	Now, Lemma \ref{2} implies the following.
	\begin{corollary}\label{3} 
		Let $ p>3 $ be a prime. Let $(A, +, \circ)$ be a brace whose multiplicative group is the group XV. Then \[Q*A^{i}\subseteq A^{i+2},\ Q\in A^{2}.\]
		Moreover, \[ Q*Q\in Q*A^{2}\subseteq A^{4}.\]
	\end{corollary}
	Furthermore, using the relation in group XV \[P=Q^{-1}\circ S^{-1}\circ Q\circ S,\] Lemma \ref{2}, and Corollary \ref{3} we have the following corollary.
	\begin{corollary}\label{4} 
		Let $(A, +, \circ)$ be a brace whose multiplicative group is the group XV. Then  \[P*A^{i}\subseteq A^{i+3},\ P\in A^{3}.\]
		Therefore, \[P*P=0,\ Q*P=0,\ P*Q=0,\ A*P= P*A\subseteq A^{4},\ P*A^{2}=A^{2}*P=0.\]
	\end{corollary}

	\begin{proposition}\label{7} 
		Let $ p>3 $ be a prime. Let $(A, +, \circ)$ be a brace which is not right nilpotent with multiplicative group XV. Then $A^{4}\neq 0$ and \[P\notin A^{4}.\]
		Moreover, $A^{2}$ is a brace of cardinality $p^{3}$ and $A^{3}$ is a brace of cardinality $p^{2}$, so \[A^{3}=\mathbb F_{p}P+A^{4}.\]
	\end{proposition}
	\begin{proof}
		Observe first  that $P*a=a*P$ for any $a\in A$ since $P$ is central in $(A, \circ )$. Recall $P\in A^{3}$, so $P*A=A*P\subseteq A^{4}$. Therefore, $A^{4}=0$ implies that $I=\mathbb F_{p}P$ is an ideal in $A$ and $I*A=A*I=0$. Notice that $A/I$ has cardinality $p^{3}$, and it is known that this implies that $A/I$ is right nilpotent \cite{DB}. Now $A*I=I*A=0$ implies that $A$ is right nilpotent. This contradicts the assumption that $A$ is not right nilpotent. 
		
		We will now show that $P\notin A^{4}$. If $P\in A^{4}$, then $A*P\subseteq A^{5}=0$, and since $P$ is central, we have $P*A=0$. Therefore, reasoning similarly to above, $I=\mathbb F_{p}P$ is an ideal in $A$ and $A/I$ is right nilpotent. Therefore $P\notin A^{4}$.
		
		Notice that, by Corollary \ref{5}, $A^{4}\neq 0$ implies $A^{i}/A^{i+1}$ is one dimensional vector space for $i=1,2,3,4$. It follows that $A^{2}$ is a brace of cardinality $p^{3}$ and $A^{3}$ has cardinality $p^{2}$, so $A^{3}=\mathbb F_{p}P+A^{4}$.
	\end{proof} 
	\begin{proposition}\label{aQ} 
		Let $ p>3 $ be a prime. Let $(A, +, \circ)$ be a brace which is not right nilpotent with multiplicative group XV and let $a\in A^{4}$. Then \[ a*Q=P^{\alpha},\] for some $1\leq \alpha \leq p$. 
	\end{proposition}
	\begin{proof}
		Consider the element $e= a^{-1}\circ Q^{-1}\circ a\circ Q$ where the inverses are taken in the group $(A, \circ)$. Notice that since $Q*A^{i}\subseteq A^{i+2}$ and by Lemma \ref{2} we get \[e\in A^{3},\  e*A^{i}\subseteq A^{i+3}.\]
		For $r\in A$ denote \[E_{r}= e^{-1}\circ r^{-1}\circ e \circ  r.\] 
		Now using similar reasoning as before we get \[ E_{r}\in A^{4},\ E_{r}*A\subseteq A^{5}=0. \]
		Since $E_{r}\in A^{4}$, then $A*E_{r}=0$. This implies that $\mathbb F_{p}E_{r}$ is an ideal in $A$. Therefore, if $E_{r}\neq 0$, then $A/\mathbb F_{p}E_{r}$ is a brace of cardinality $p^{3}$ or less, and hence it  is right nilpotent by \cite{DB}. This implies that $A$ is right nilpotent. Since we only consider not right nilpotent braces this is impossible. Therefore, $E_{r}=0$. This holds for every $r\in R$, so it follows that $e$ is in the centre of group $(A, \circ )$. Because the group $(A, \circ )$ is group $XV$ it follows that $e=P^{i}$ for some  natural number $i$. Therefore, by writing $e$ as sums of products of $Q$ and $a$ and using the fact that $A*a=0$ since $a\in A^{4}$ we get $ a*Q=P^{\alpha}$, for some $0\leq \alpha <p$ (since $A^{3}*A^{3}=A^{3}*(\mathbb F_{p}P+A^{4})=P*A^{3}=0$). 
	\end{proof} 
	\begin{proposition}\label{6} 
		Let $ p>3 $ be a prime. Let $(A, +, \circ)$ be an $\mathbb F_{p}$-brace which is not right nilpotent with multiplicative group XV. Suppose that  $R\notin A^{2}$ and $Q\notin A^{3}$. Then there exist element $\bar S\in A^{4}$ such that the mapping $f: (A, \circ )\rightarrow (A, \circ)$ given by $f(P)=P$, $f(Q)=Q$, $f(R)=R$, $f(S)=\bar S$ is an automorphism of the group $(A, \circ )$. 
	\end{proposition}
	\begin{proof} 
		Observe that the mapping $f:G\longrightarrow G$ of the group $G=(A, \circ )$ defined on generators of $G$ by \[f(R)=R,\ f(Q)=Q,\ f(P)=P,\ f(S)=S \circ  R^{i} \circ Q^{j} \circ P^{k}\] is a group homomorphism for any $i, j, k$. 
		
		Observe that since $R\notin A^{2}$, then for some $i$ the element  $S_2:=S\circ R^{i}$ will belong to $A^{2}$ (because by Corollary \ref{4} we have that $A/A^{2}$ has dimension $1$ as an $\mathbb F_{p}$-vector space). Notice also that since  $Q\notin A^{3}$, then for some $j$ element $S_{3}:=S_{2}\circ Q^{j}$ will belong to $A^{3}$. 
		
		We know that $P\in A^{3}$ and $P\notin A^{4}$. Consequently, for some $k$ we have that element  $S_{4}=S_{3}\circ P^{k}$  will be in  $A^{4}$.  Therefore $S\circ R^{i}\circ Q^{j}\circ P^{k}\in A^{4}$. We can now define $\bar S=S\circ R^{i}\circ Q^{j}\circ P^{k}$.
		
		It remains to show that $f$ is an automorphism, so the kernel of $f$ is trivial. This follows because the image of a non-trivial element $P^{\alpha }\circ Q^{\beta }\circ R^{\gamma } \circ S^{\xi }$ will be a non-trivial element (alternatively, it is easy to give the formula for the inverse map $f^{-1}$).  
	\end{proof}

%
	\subsection{The Case when $R\in A^{2}$}
	\begin{lemma}\label{11} 
		Let $ p>3 $ be a prime. Let $(A, +, \circ)$ be an $\mathbb F_{p}$-brace whose multiplicative group is the group XV. If $R\in A^{2}$, then $A^{2}$ is a commutative brace. 
	\end{lemma}
	\begin{proof} 
		Suppose that $R\in A^{2}$ so  
		$P, Q, R\subseteq A^{2}$. Since the set $\{P^{i}\circ Q^{j}\circ R^{t}: 0<i,j,t \leq p\}$ has cardinality $p^{3}$ and is contained in $A^{2}$ it follows that $A^{2}=\{P^{i}\circ Q^{j}\circ R^{t}: 0<i,j,t\leq p\}$, by Proposition \ref{7}. Therefore, $(A^{2}, \circ)$ is commutative (since $P, Q, R$ commute with each other). It follows that $(A, +, *)$ is a commutative ring by \cite{cjo}.
	\end{proof}

	\begin{lemma}\label{12} 
		Let $(A, +, \circ)$ be a brace which is not right nilpotent with multiplicative group XV. If $R\in A^{2}$, then $Q*Q=0$.
	\end{lemma}
	\begin{proof} 
		Notice that from the relation \[R\circ S=S\circ R\circ Q\] we get \[R*S=S*R+S*Q+R*Q+Q.\]
		therefore, 
		\begin{align}\label{E2}
			Q\ast\left(R*S\right)&=Q\ast\left(S*R\right)+Q\ast\left(S*Q\right)+Q\ast\left(R*Q\right)+Q\ast Q\nonumber \\
			Q\ast\left(R*S\right)&=Q\ast\left(S*R\right)+Q\ast Q \in A^{4} \ \text{since} \ Q \in A^{2}.
		\end{align}
		We have also the group relation $Q\circ R=R\circ Q$, therefore \[(Q\circ R)*S=(R\circ Q)*S,\]
		which using (\ref{BR}) implies that $ Q*(R*S)=R*(Q*S)$. Now \[ Q*S \in A^{3}= \mathbb F_{P}P+A^{4},\] so $ R*(Q*S) \in \mathbb F_{P}P\ast R=0 $. Thus substituting in (\ref{E2}) we find \[ 0=R*(Q*S)=Q\ast\left(R*S\right)=Q\ast\left(S*R\right)+Q\ast Q,\] 
		since we have assumed $ R \in A^{2} $, we have that $ Q\ast\left(S*R\right)=0 $, so $ Q\ast Q=0 $. 
	\end{proof}


	\begin{proposition}\label{R} 
		Let $(A, +, \circ)$ be an $\mathbb F_{p}$-brace which is not right nilpotent with multiplicative group XV. If $R\in A^{2}$, then $Q\in A^{3}$.
	\end{proposition}
	\begin{proof} 
		Suppose on the contrary that $Q\notin A^{3}$. Recall that, by Corollaries \ref{3} and \ref{4} and Proposition \ref{7}, $Q\in A^{2}$, $P\in A^{3}, P\notin A^{4}$. Hence,  
		\[A^{2}=\mathbb F_{p}Q+\mathbb F_{p}P+A^{4}.\]
		In group XV we have the relation
		\[Q\circ S=S\circ Q\circ P,\]
		hence $Q*S=S*Q+S*P+P$, by Corollary \ref{4}. We also have the relation $R\circ S=S\circ R\circ Q,$
		so \[R*S=S*R+S*Q+ R*Q+Q\] (notice that $S*(R*Q)=S*(Q*R)\in S*A^{4}\subseteq A^{5}=0$). Also, we have the relation    
		\[(R\circ S)*S=(S\circ R\circ Q)*S,\]
		hence \[R*(S*S)=S*(R*S)+S*(Q*S)+Q*S.\]
		
		Observe that $R*(S*S)\subseteq A^{2}*A^{2}$. Recall that $A^{2}$ is a commutative ring by Lemma \ref{11} and $A^{2}=\mathbb F_{p}Q+\mathbb F_{p}P+A^{4}$, hence \[A^{2}*A^{2}=A^{2}*(\mathbb F_{p}Q+\mathbb F_{p}P+A^{4})\subseteq {\mathbb F}_{p}Q*Q+\mathbb F_{p}A^{2}*P+ A^{2}*A^{4}=0\] (since $Q*Q=0$ by Lemma \ref{12} and $A^{2}$ is commutative). 
		Next, by the above we have \[ S*(R*S)=S*(S*R+S*Q+R*Q+Q).\]
		Therefore, \[S*Q+Q*S=-(S*(S*R)+S*(S*Q)+S*(Q*S))\subseteq A^{4}.\] 
		Recall from the beginning of this proof that $Q*S-S*Q=S*P+P$, so \[S*Q\in {\frac {-P}2}+A^{4}.\] 
		Therefore, $Q*S\in {\frac P2}+A^{4}$. 
		
		Recall the relation \[(Q\circ S)*S=(S\circ Q\circ P)*S.\]
		Hence $Q*(S*S)=S*(Q*S)+P*S$, and by the above \[Q*(S*S)\subseteq Q*A^{2}\subseteq A^{2}*A^{2}=0,\] and $S*(Q*S)\in S*({\frac P2}+A^{4})\in {\frac {S*P}2}$. Therefore, ${\frac 32}S*P=0$. Hence $P*A=A*P=0$ so $\mathbb F_{p}P$ is an ideal in $A$. Consequently, $A/\mathbb F_{p}P$ is right nilpotent, so $A$ is right nilpotent, a contradiction. 
	\end{proof}
	
	\subsection{The Case when $Q\in A^{3}$}
	\begin{theorem}\label{Q} Let $p>3$ be a prime number. Let $A$ be an $\mathbb F_{p}$-brace whose multiplicative group $(A, \circ )$ is the group XV. Suppose that $Q\in A^{3}$. Then $A$ is a right nilpotent brace. 
	\end{theorem}
	\begin{proof}  
		Notice that the set of all products $P^{i}\circ Q^{j}$ has cardinality $p^{2}$, and we know from Proposition \ref{7} that $A^{3}=A* (A* A)$ has cardinality $p^{2}$. Moreover, $P^{i}\circ Q^{j}\in A^{3}$ (where $P^{j}=P\circ \cdots \circ P$, where $P$ appears $j$ times in this product). So  $A^{3}=\{P^{i}\circ Q^{j}: i,j\leq p\}$. Since $A^{4}\subseteq A^{3}$, it follows that there are integers $i,j$ and $0\neq E\in A^{4}$ such that 
		\[E=P^{i}\circ Q^{j}.\]
		
		We will now consider two cases.
		
		{\bf Case 1.} Either $S\in A^{2}$ or both $S, R\notin A^{2}$. Because $A/A^{2}$ has dimension $1$. Then there is $0<j\leq p$ such that 
		\[S\circ R^{j}\in A^{2}.\] Denote $Z=S\circ R^{j}$. 
		
		By the group relations we  have \[Q\circ Z=Z\circ Q\circ P.\] Recall that $Q*P=0$, so \[Q*Z=Z*Q+Z*P+ P.\]
		Notice $Q*Z\in Q* A^{2}\subseteq A^{4}$ and $Z*Q\subseteq Z*A^{3}\subseteq A^{4}$, $Z*P\subseteq Z*A^{3}\subseteq A^{4}$, so \[P\in A^{4}.\]
		But $P\in A^{4}$ implies that $A$ is nilpotent, as shown previously. 
		
		{\bf Case 2.} Suppose $S\in A, S\notin A^{2}$ and $R\in A^{2}$. 
		
		Observe that the group relation \[Q\circ S=S\circ Q\circ P\]
		implies \[Q*S=S*Q+S*P+P,\] since $Q*P=0$ by Corollary \ref{4}.
		Therefore, \[S*(Q*S)=S*(S*Q)+S*(S*P)+S*P=S*P\] since \[S*(S*Q)\subseteq S*(S*A^{3})\subseteq A^{5}=0.\] 
		Observe now that we also have a relation
		\[(Q\circ S)*S=(S\circ Q\circ P)*S,\]
		which implies $Q*(S*S)=S*(Q*S)+P*S.$ By the above $Q*(S*S)=S*(Q*S)+P*S=2S*P$ (since $P*S=S*P$). Notice that if $S*P=0$ then  $P*A=A*P=0$ and similarly as in previous section it implies that $A$ is right nilpotent. If $S*P\neq 0$, then $Q*(S*S)\neq 0$. Notice that $A^{2}=\mathbb F_{p}R+\mathbb F_{p}Q+\mathbb F_{p}P +A^{4}$. Therefore, \[Q*(S*S)\subseteq Q*A^{2}=\mathbb F_{p}Q*R+\mathbb F_{p}Q*Q +\mathbb F_{p} Q*P=\mathbb F_{p}Q*R,\] since $Q\in A^{3}$. It follows that $Q*R\neq 0$.
		
		Note that because $R\circ Q=Q\circ R$ we get $Q*R=R*Q\neq 0$. Observe that by the above there are $i,j$ such that  $Q^{i}\circ P^{j}=E\in A^{4}$, where $i$ is a natural number not divisible by $ p $, since $P\notin A^{4}$, as $ P\in A^{4} $ would imply that $ A $ is a right nilpotent brace (for example by the first Lemma in Section \ref{S8}). Observe that it follows that $E*R=(Q^{i}\circ P^{j})*R=i(Q*R)\neq 0$ (since $P*R\subseteq P*A^{2}=0$). On the other hand, since $R,P,Q\in A^{2}$ and $A^{2}$ has cardinality $p^{3}$ then $A^{2}$ is commutative. Therefore, $E*A^{2}=A^{2}*E\subseteq A^{2}*A^{4}\subseteq A^{5}=0$, which is a contradiction. Notice that $ R \in A^{2} $ so $ E*A^{2}=0 $ gives $ E*R=0 $. However, we have shown that $ E*R $ is not zero above. This gives a contradiction.
	\end{proof}
	
%

	\subsection{The Case when $S\in A^{4}$}
	\begin{proposition}\label{p} Let $p>3$ be a prime number. Let $A$ be an $\mathbb F_{p}$-brace whose multiplicative group $(A, \circ )$ is the group XV. Then $(A, \circ)$ is generated by elements $P, Q, R, S$ which satisfy relations from group XV. Moreover, the following relations hold for $a\in \{P, Q, R, S\}$.
		\begin{enumerate}
			\item \[Q\circ S=S\circ Q\circ  P,\ (Q\circ S)*a=(S\circ Q\circ  P)*a\]
			\item \[Q\circ R=R\circ Q,\ (Q\circ R)*a=(R\circ Q)*a\]
			\item \[P\circ S=S\circ P,\ P\circ Q=Q\circ P,\ P\circ R=R\circ P,\]
			\[(P\circ S)*a=(S\circ P)*a=0,\ ( P\circ Q)*a=(Q\circ P)*a=0,\ (P\circ R)*a=(R\circ P)*a\]
			\item \[ R\circ S=S\circ R\circ Q,\ (R\circ S)*a=(S\circ R\circ Q)*a\]
			\item \[P^{p}=Q^{p}=R^{p}=S^{p}=0,\ (P^{p})*a=(Q^{p})*a=(R^{p})*a=(S^{p})*a=0\]
		\end{enumerate}
	\end{proposition}
	\begin{proof} 
		It follows from the fact that group XV is the multiplicative group $(A, \circ )$ of our brace. Recall that $0$ is the identity element of the additive group $(A, +)$ and also the identity element of the multiplicative group $(A, \circ )$ (these identity elements coincide in every brace).  
	\end{proof}

	\begin{proposition}\label{listings} 
		Let $p>3$ be a prime number. Let $(A, +, \circ)$ be an $\mathbb F_{p}$-brace which is not right nilpotent with multiplicative group XV. Then $(A, \circ)$ is generated by elements $P, Q, R, S$ which satisfy relations from group XV. In addition, we have the following.
		\begin{enumerate}
			\item $S\in A^{4},$ $P\in  A^{3}, P\notin A^{4}$, $Q\in A^{2}$. Moreover, $Q*A^{i}\subseteq A^{i+2}$, $P*A^{i}\subseteq A^{i+3}$. 
			\item $Q\notin A^{3}$,  $R\in A$, $R\notin A^{2}$. 
			\item $Q*P=0$, $P*P=0$, $S*P=0$, $R*P=yS$ for some $y\in \mathbb F_{p}$ with $ y\neq 0 $.
			\item $Q*S=0$, $P*S=0$, $S*S=0$, $R*S=0$.  
			\item $Q*Q=\alpha S$, $P*Q=0$, $S*Q=\gamma P$, $R*Q=z_{1}P+zS$ for some $\alpha, \gamma, z, z_{1} \in \mathbb F_{p}$.
			\item $Q*R=z_{1}P+zS$, $P*R=yS$ for some $z, z_{1}, y \in \mathbb F_{p}$ with $y\neq 0$.
			\item $S*R=j'Q+i'P+k'S$, $R*R=jQ+iP+kS$ for some $j, j',i,  i', k, k'\in \mathbb F_{p}$.
			\item $\gamma =-1.$
			\item $\alpha =-y $.
			\item $j'=-1$. 
			\item $z_{1}=0$ and $j=0$.
			\item $2z=y$, $i'=1$, $k'=-z$.
		\end{enumerate}
	\end{proposition}
	\begin{proof}
		\begin{enumerate}
			\item It follows from  Corollary \ref{3}, Proposition \ref{7}, Proposition \ref{6}.
			\item It follows from Proposition \ref{R} and  Theorem \ref{Q}.
			\item It follows from Corollary \ref{3} that $Q*P, P*P, S*P\in A^{5}=0$ (since $S*P=P*S$ as $S\circ P=P\circ S$). Moreover,  $R*P\in A^{4}=\mathbb F_{p}S$, which gives the conclusion.
			\item It follows from the fact that $S\in A^{4}$ so $A*S=0$.  
			\item By Corollary \ref{3} we have that $P*Q\in A^{5}=0$, $Q*Q\in A^{4}=\mathbb F_{p}S$. Similarly, since $Q\in A^{2}$, then $R*Q\in A^{3}=\mathbb F_{p}P+\mathbb F_{p}S$, by item 1 above. The fact that $S*Q=\gamma P$ follows from Proposition \ref{aQ}. 
			\item Since $P$ is central in the group XV it follows that $P*R=R*P=yS$ by item 3 above.  Similarly $Q*R=R*Q$ since $Q\circ R=R\circ Q$, and the result follows from item 5 above.
			\item Notice that $S*R, R*R\in A^{2}=\mathbb F_{p}Q+\mathbb F_{p}P+\mathbb F_{p}S$ by item 1.
			\item Consider relation
			\[ S\circ Q\circ P=Q\circ S.\]
			It can be rewritten as $S*Q+P+S*P=Q*S=0$ since $Q*S\subseteq A*A^{4}\subseteq A^{5}=0$. This gives $\gamma =-1$.
			\item Consider relation \[(R\circ S)*Q=(S\circ R\circ Q)*Q.\]
			We get $Q*Q+S*(R*Q)=R*(S*Q)$, hence $\alpha =-y$.
			\item Consider the relation $ (S\circ Q\circ P)*R=(Q\circ S)*R$ which is an implication of one of the group relations in group XV. We can rewrite it as \[S*(Q*R)+P*R=Q*(S*R).\]
			This implies $-y=-\alpha j'$. We know that $y\neq 0$ since $A$ is not right nilpotent, and $\alpha =-y$, hence $j'=-1$.
			\item We will first show that  $j=-z_{1}$. Consider the group relation $Q\circ R=R\circ Q$, it implies $(Q\circ R)*R=(R\circ Q)*R$. This implies $Q*(R*R)=R*(Q*R)$, hence $j\alpha =z_{1}y$. Since $y\neq 0$ and $y=-\alpha $ we get $z_{1}=-j$.
			
			We will now to show that $z_{1}=0$. Consider the relation $(S\circ R\circ Q)*R=(R\circ S)*R.$ This implies $S*(R*R)+Q*R+R*(Q*R)=R*(S*R)$ since $S*(Q*R)\subseteq A^{5}=0$. We can consider the component corresponding to $P$ in this equation. This implies $-j+z_{1}=j'z_{1}$, and since $j'=-1$ we get $j=2z_{1}$. Since we previously showed that $z_{1}=-j$ it follows that $z_{1}=j=0$.
			\item Consider the equation $S*(R*R)+Q*R+R*(Q*R)=R*(S*R)$ from the previous sub-point. Comparing the component at $S$ in this equation we get  $z=zj'+ i'y$, hence $2z=i'y$.  
			
			We will now show that $i'=1$ hence $2z=y$. We will now consider the relation \[S\circ R\circ Q=R\circ S.\]
			This implies \[S*R+S*Q+R*Q+Q=0\] since $S*(R*Q)\subseteq S*A^{3}=0$ and $R*S\subseteq A*A^{4}=A^{5}=0$. Comparing the component at $P$ in this equation we get $i' -1=0$, hence $i'=1$, as required. Comparing component at $S$ we get $k'+ z=0$ hence $k'=-z$. 
		\end{enumerate}	
	\end{proof}

	\begin{remark}
		A general comment on Proposition \ref{listings} to explain the approach. In first subpoint we show that \[A=F_{p}P+F_{p}R+F_{p}S+F_{p}Q\]
		and then so we can write each product, for example $ Q*Q $ as $ iP+jR+kS+lQ $ for some $ i,j,k,l $ in $ \mathbb{F}_{p} $. Next we substitute it to the equations, which need to hold in the brace, listed in Theorem \ref{Q}, and then list the consequences.
	\end{remark}

	\section{The Braces are Well-defined}
	\subsection{The Map $\lambda $}
	We now consider how our operation $*$ looks in our brace. In the previous section we showed that if $(A, +, \circ )$ is brace which is not right nilpotent whose group $(A, \circ )$ is the group XV, then  
	\[ P*P=0,\ P*Q=0 ,\ P*R=yS ,\ P*S=0,\]
	\[ Q*P=0,\ Q*Q=-yS,\ Q*R=2^{-1}yS,\ Q*S=0,\]
	\[ R*P=yS,\ R*Q=2^{-1}yS,\ R*R=iP+kS,\ R*S=0,\]
	\[ S*P=0,\ S*Q=-P,\ S*R=-Q+P-2^{-1}yS,\ S*S=0.\]
	where $ y,i,k \in \mathbb{F}_{p} $ with $ y\neq0 $.
	
	
	Therefore, we get the formulas for the $\lambda $ maps in $A$, where $\lambda_{a}(b)=a*b+b$. 
	\[\lambda_{P}(P)=P,\ \lambda_{P}(Q)=Q,\ \lambda_{P}(R)=R+yS,\ \lambda_{P}(S)=S,\]
	\[\lambda_{Q}(P)=P,\ \lambda_{Q}(Q)=Q-yS,\ \lambda_{Q}(R)=R+2^{-1}yS,\ \lambda_{Q}(S)=S,\]
	\[\lambda_{R}(P)=P+yS,\ \lambda _{R}(Q)=Q+2^{-1}yS,\ \lambda _{R}(R)=R+iP+kS,\ \lambda _{R}(S)=S,\]
	\[\lambda_{S}(P)=P,\ \lambda _{S}(Q)=Q-P,\ \lambda_{S}(R)=R-Q+P-2^{-1}yS,\ \lambda_{S}(S)=S.\]
	
	Let $(B', +)$ be the abelian group generated by elements $Q', P', R', S', 1$  where every element has order $p$, so $pQ'=pQ'=pR'=pS'=p1=0$. Therefore, \[B'=\mathbb F_{p}P'+\mathbb F_{p}Q'+\mathbb F_{p}R'+\mathbb F_{p}S'+\mathbb F_{p}1,\] where $1$ is some element of $B$.
	
	Let $(B, +)$ be the subgroup of $(B', +)$ generated by elements $P', Q', R', S'$. Define the homomorphisms $\lambda_{P}, \lambda_{Q}, \lambda _{R}, \lambda_{S}$ from $(B', +)$  to $(B', +)$ by using the table for $\lambda $ maps in $B$ above, and by setting $\lambda _{g}(1)=g'+1$ for $g\in \{P, Q, R, S\}$.  
	
	So we have the following relations which depend on parameters $i, k, y\in \mathbb F_{p}$ with $y\neq 0$ given below.
	\begin{adjustwidth}{-1.5cm}{}
	\begin{align}\label{Table1}
		\lambda_{P}(P')&=P',\ \lambda_{P}(Q')=Q',\ \lambda _{P}(R')=R'+yS',\ \lambda_{P}(S')=S',\ \lambda_{P}(1)=P'+1,\\
		\lambda_{Q}(P')&=P',\ \lambda_{Q}(Q')= Q'-yS',\ \lambda _{Q}(R')=2^{-1}yS'+R',\ \lambda_{Q}(S')=S',\ \lambda _{Q}(1)=Q'+1,\nonumber\\
		\lambda_{R}(P')&=P'+yS',\ \lambda_{R}(Q')= Q'+2^{-1}yS',\ \lambda _{R}(R')=iP'+kS'+R',\ \lambda_{R}(S')=S',\ \lambda _{R}(1)=R'+1,\nonumber\\
		\lambda_{S}(P')&=P',\ \lambda _{S}(Q')= Q'-P',\ \lambda _{S}(R')=R' -Q'+P'-2^{-1}yS', \lambda _{S}(S')=S', \lambda _{S}(1)= S'+1.\nonumber
	\end{align}
	\end{adjustwidth}
	
	\begin{theorem}\label{17} Let $p>3$ be a prime number and $i,k,y\in \mathbb F_{p}$ with $y\neq 0$. Let $(B', +)$ and $(B, +)$ be defined as above. Then both $(B', +)$ and $(B, +)$ are linear spaces over the field $\mathbb F_{p}$. Let linear maps $\lambda_{g}: (B', +)\longrightarrow (B',+)$ for $g\in \{R, Q, P, S\}$ be defined by lines 1-4 of (\ref{Table1}) above. Then the following relations hold for any $a\in B'$.
		\begin{enumerate}
			\item \[\lambda_{Q}(\lambda_{S}(a))=\lambda_{S}(\lambda_{Q}(\lambda_{ P}(a))),\]
			\item \[\lambda_{Q}(\lambda_{R}(a))=\lambda_{R}(\lambda_{Q}(a)),\]
			\item \[\lambda_{P}(\lambda_{S}(a))=\lambda_{S}(\lambda_{P}(a)),\ \lambda_{P}(\lambda_{Q}(a))=\lambda_{Q}(\lambda_{P}(a)),\ \lambda_{P}(\lambda_{R}(a))=\lambda_{R}(\lambda_{P}(a)),\]
			\item \[\lambda_{R}(\lambda_{S}(a))=\lambda_{S}(\lambda_{R}(\lambda_{Q}(a))),\]
			\item \[\lambda_{P}^{p}(a)=\lambda_{Q}^{p}(a)=\lambda _{R}^{p}(a)=\lambda _{S}^{p}(a)=a.\]
		\end{enumerate}
		Therefore, the maps $\lambda_{g}: (B', +)\longrightarrow (B',+)$ defined inductively using the formula 
		\[\lambda _{g\circ g'}(a)=\lambda _{g}(\lambda _{g'}(a))\]
		for $g, g'\in (G, \circ )$ and $a\in (B', +)$ are well defined automorphisms of $B'$. Moreover, maps $\lambda_{g}:B\longrightarrow B$ with the domain restricted to $B$ are  automorphisms of $(B, +)$.
	\end{theorem}
	

	\begin{proof} To check these properties we will present the $\lambda $ maps in matrix form, and the above properties can be checked by multiplying the matrices $M_{g}$. Write the matrices of the linear maps $\lambda_{g}$ for $g\in \{R, Q, P, S\}$ in the base $e_{1}=1, e_{2}=R', e_{3}=Q', e_{4}=P', e_{5}=S'$ where $e_{i}$ denotes the vector with all entries zero except of the $i$-th  entry which is $1$. Therefore, we have matrices $M_{i}$ of dimension $5$ such that $\lambda_{g}(e_{j})=M_{g}e_{j}$ for each $g\in \{R,Q,P, S\}$ where  $e_{1}=1, e_{2}=R', e_{3}=Q', e_{4}=P', e_{5}=S'$. 
		
	We have the matrices
		\begin{align*}
			M_{P}&=\begin{bmatrix}
				1&0&0&0&0 \\
				0&1&0&0&0 \\
				0&0&1&0&0\\
				1&0&0&1&0\\
				0&y&0&0&1
			\end{bmatrix} \ &M_{Q}=\begin{bmatrix}
				1&0&0&0&0\\
				0&1&0&0&0\\
				1&0&1&0&0\\
				0&0&0&1&0\\
				0&2^{-1}y&-y &0&1\\
			\end{bmatrix}\\
			M_{R}&=\begin{bmatrix}
				1&0&0&0&0\\
				1&1&0&0&0\\
				0&0&1&0&0\\
				0&i&0&1&0\\
				0&k&2^{-1}y&y&1\\
			\end{bmatrix}\ &M_{S}=\begin{bmatrix}
				1&0&0&0&0\\
				0&1&0&0&0\\
				0&-1&1&0&0\\
				0&1&-1&1&0\\1&-2^{-1}y&0&0&1\\
			\end{bmatrix}
		\end{align*}
		observe now that, by direct calculation, we have 
		\[M_{Q}M_{S}=M_{S}M_{Q}M_{P},\ M_{Q}M_{R}=M_{R}M_{Q},\]
		\[M_{P}M_{S}=M_{S}M_{P},\ M_{P}M_{Q}=M_{Q}M_{P}, \ M_{P}M_{R}=M_{R}M_{P},\]
		\[M_{R}M_{S}=M_{S}M_{R}M_{Q},\]
		\[M_{P}^{p}=M_{Q}^{p}=M_{R}^{p}=M_{S}^{p}=Id,\]
		where $Id$ is the identity matrix. This proves the first part of our theorem.
		
		We can now define maps $\lambda_{g}: (B', +)\longrightarrow (B', +)$ inductively, using the formula \[\lambda_{g\circ g'}(a)=\lambda_{g}(\lambda_{g'}(a)),\]
		for $g, g'\in (G, \circ )$ and $a\in (B', +)$. Observe that the maps $\lambda _{g}$ are group homomorphisms of $(B', +)$ so they are linear maps over the field $\mathbb F_{p}$,  so $\lambda_{g}(a+b)=\lambda _{a}(b)+\lambda _{g}(b)$. This is because they are defined as compositions of maps $\lambda_{P}, \lambda_{Q}, \lambda_{R}, \lambda_{S}$ which are homomorphisms of $(B', +)$, so 
		\begin{align*}
			&\lambda_{P}(a+b)=\lambda_{P}(a)+\lambda_{P}(b),\ \lambda_{Q}(a+b)=\lambda_{Q}(a)+\lambda_{Q}(b),\\ &\lambda_{R}(a+b)=\lambda_{R}(a)+\lambda_{R}(b),\ \lambda_{S}(a+b)=\lambda_{S}(a)+\lambda_{S}(b).
		\end{align*}
		
		Notice that maps $\lambda_{P}, \lambda_{Q}, \lambda_{R}, \lambda_{S}$ are linear maps and can be represented by matrices, therefore their  compositions are also linear maps. Consequently, for any $g, h, j\in G$ \[\lambda_{(g\circ h)\circ j}=\lambda_{g\circ h}\lambda_{j}=(\lambda_{g}\lambda_{h})\lambda _{j}=\lambda_{g}(\lambda_{h}\lambda_{j})=\lambda_{g\circ (h\circ j)}.\]
		
		We will now show that maps $\lambda_{g}$ are automorphisms of $B'$.  Notice that the images of maps $\lambda_{P}, \lambda_{Q}, \lambda_{R}, \lambda_{S}$ are subsets of $(B', +)$, and that the kernels of these maps are zero. Therefore, the maps $\lambda_{P}, \lambda_{Q}, \lambda _{R}, \lambda_{S}$ with the domain restricted to $(B,+)$ as maps from $(B, +)$ to $(B, +)$ are automorphisms of $(B,+)$. It follows that maps $\lambda_{g}$ with the  domain restricted to $(B, +)$ for $g\in G$ are also automorphisms of $(B, +)$ because they are compositions of maps $\lambda_{P}, \lambda_{Q}, \lambda_{R}, \lambda_{S}$.
	\end{proof}
	
	
	\subsection{The Map f}
	Let $B, B'$ and maps $\lambda_{g}: (B',+)\longrightarrow (B',+)$  be defined as in Theorem \ref{17} for some $i,k,y\in \mathbb F_{p}$ with $y\neq 0$ and  $p>3$. We will consider $(B, +)$ as a set $B$. Define the map $f:G\longrightarrow B$ by \[f(g)=\lambda_{g} (1)-1,\]
	for $g\in (G, \circ )$. 

	\begin{theorem}\label{map} 
		Let $B, B'$ and maps $\lambda_{g}: (B',+)\longrightarrow (B',+)$  be defined as in Theorem \ref{17} for some $i,k,y\in \mathbb F_{p}$ with $y\neq 0$ and $p>3$. Then the map $f:G\longrightarrow B$ defined as
		\[f(g)=\lambda_{g} (1)-1\]
		for $g\in (G, \circ )$ is a bijective function from $G$ to $B$. 
	\end{theorem}
	\begin{proof} 
		We will now show that $f$ is a bijective function. Observe that groups $G$ and $(B, +)$ have the same cardinality, so to show that $f$ is a bijective function it suffices to show that $f$ is injective.
		
		Suppose on the contrary that $f$ is not bijective. Then we would have $f(g)=f(h)$ for some $g,h\in G$. This implies 
		$\lambda_{g}(1)=\lambda_{h}(1)$. Let $g'= h^{-1}\circ g$ so $\lambda _{g'}(1)=1$. Every element of $G$ can be written as a product of elements of some powers of elements $P, Q, R, S$. Therefore, there are some integers $0\leq \alpha , \beta , \gamma , \xi < p$ such that 
		\[g'= R^{\gamma}\circ Q^{\beta}\circ P^{\alpha}\circ S^{\xi}.\]
		
		Notice that it follows from Theorem \ref{17}  (using the fact that $\lambda _{a\circ b}(c)=\lambda_{a}\lambda_{b}(c)$ several times) that $\lambda_{  R^{\gamma}\circ Q^{\beta}\circ P^{\alpha}\circ S^{\xi}}(1)-1=\gamma R'+C$ where $C\in \mathbb F_{p}Q'+\mathbb F_{p}P'+\mathbb F_{p}S'$, hence \[\lambda _{g'}(1)-1=\gamma R'+C.\] Notice that $ \lambda_{g'}(1)=1$ implies $\gamma =0$, so \[g'=Q^{\beta}\circ P^{\alpha}\circ S^{\xi}.\]
		
		Notice that it follows from Theorem \ref{17}  (using the fact that $\lambda _{a\circ b}(c)=\lambda _{a}\lambda_{b}(c)$ several times) that $\lambda_{ Q^{\beta}\circ P^{\alpha}\circ S^{\xi}}(1)-1=\beta Q'+D$ where $D\in \mathbb F_{p}P'+\mathbb F_{p}S'$, hence \[\lambda_{g'}(1)-1=\beta Q'+D.\] Notice that $ \lambda_{g'}(1)=1$ implies $\beta =0$, so \[g'= P^{\alpha}\circ S^{\xi}.\]
		
		Observe now that by Theorem \ref{17} we get $\lambda_{g'}(1)=\lambda_{P^{\alpha}\circ S^{\xi}}(1)=\alpha P'+\xi S'+1$. Therefore, $\lambda_{g'}(1)=1$ implies $\lambda =\xi=0$, so $g'=g\circ h^{-1}$ is the identity element in $G'$, and so $g=h$ as required. 
	\end{proof}
	
	\subsection{Bijective Cocycles and Braces}\label{BC}
	\begin{definition}\label{mapf} Let $(G, \circ)$ be a group, $(H, +)$ be an abelian group, $p:G\longrightarrow\mathrm{Aut}(H)$ be an action and $f:G\longrightarrow  H$ be such that for $g,h\in G$ we have 
		\[ f( g \circ h)=f(g)+ p_{g}(f(h)).\]
	Then $f$ is a $1$-cocycle with respect to $p_{g}$. 
	\end{definition}
	The following result is known (cf. Jespers, Ced{\' o}, Okninski, Rio \cite[Theorem 2]{Rio}).
	
	\begin{theorem}[Theorem 2, \cite{Rio}]\label{cocycles}
		Let $(G, \circ)$ be a group, $(H, +)$ be an abelian group, $p_{g}$  be an automorphism of $H$ for each $g$ in $G$, and $f: G\longrightarrow H$ be a bijective 1-cocycle with respect to $p$. Define an operation on $G$ for $g,h\in G$ by \[g + h =f^{-1}\left(f(g) +f(h)\right).\]
		Then $(G, \circ , +)$ is a brace.
	\end{theorem}
	
	\begin{theorem}\label{main} Let $p>3$ be a prime number. Let $(G, \circ)$ be the group XV, and let automorphisms of $(B,+)$, given by $\lambda _{g}:(B,+)\longrightarrow (B,+)$, be defined as in Theorem \ref{17} for fixed $y, i, k$. Then the group $(G, \circ)$ with the addition defined as \[g + h =f^{-1}\left(f(g)+f(h)\right),\]
		where $f$ is defined as in Theorem \ref{map}, is a brace. 
	\end{theorem}
	\begin{proof} 
		By Theorem \ref{cocycles} it suffices to show that the map $f$ is a bijective cocycle with respect to the map 
		\[p_{g}=\lambda_{g}.\] By Theorem \ref{map} $f$ is a bijective function. 
		To show that $f$ is a cocycle, we need to show that  
		\[ f(g\circ h)=f(g)+ p_{g}(f(h)).\]
		By using the formula 
		\[f(g)=\lambda_{g} (1)-1,\]
		the left hand side becomes $f(g\circ h)=\lambda_{g\circ h}(1)-1$ and  the right hand side becomes \[f(g)+ p_{g}(f(h))=\left(\lambda_{g}(1)-1\right) +\lambda_{g}\left(\lambda _{h}(1)-1\right)=\lambda_{g\circ h}(1)-1\]
		since $\lambda _{a\circ b}(1)=\lambda _{a}(\lambda _{b}(1))$ by the definition of $\lambda $, and $\lambda_{g}(x+y)=\lambda_{g}(x)+\lambda_{g}(y)$ since $\lambda $ is a automorphism of $(B', +)$ by Theorem \ref{17}.
	\end{proof}

	\section{All Braces}\label{allbraces}
	In this section we show that in Theorem \ref{main} we constructed all the $\mathbb F_{p}$-braces which are not right nilpotent with multiplicative group XV.
	
	We will refer to the Multiplicative Table. The following set of rules will be called the Multiplicative Table. 
	\[ P*P=0,\ P*Q=0,\ P*R=yS,\ P*S=0,\]
	\[ Q*P=0,\ Q*Q=-yS,\ Q*R=2^{-1}yS,\ Q*S=0,\]
	\[ R*P=yS,\ R*Q=2^{-1}yS,\ R*R=iP+kS,\ R*S=0,\]
	\[ S*P=0,\ S*Q=-P,\ S*R=-Q+P-2^{-1}yS,\ S*S=0.\]
	
	Notice, that using the formulas $\lambda _{a}(b)=a*b+b$ we can write the Multiplicative Table in the following form. 
	\[ \lambda_{P}(P)=P,\ \lambda_{P}(Q)=Q,\ \lambda_{P}(R)=R+yS,\ \lambda _{P}(S)=S,\]
	\[ \lambda_{Q}(P)=P,\ \lambda_{Q}(Q)=Q-yS,\ \lambda_{Q}(R)=R+2^{-1}yS,\ \lambda_{Q}(S)=S,\]
	\[ \lambda_{R}(P)=P+yS,\ \lambda_{R}(Q)=Q+2^{-1}yS,\ \lambda_{R}(R)=R+iP+kS,\ \lambda _{R}(S)=S,\]
	\[ \lambda_{S}(P)=P,\ \lambda_{S}(Q)=Q-P,\ \lambda_{S}(R)=R-Q+P-2^{-1}yS,\ \lambda _{S}(S)=S.\]
	We can obtain the Multiplicative Table by putting $a*b=\lambda _{a}(b)-b$.
	
	\begin{theorem}\label{9}
		Let $p>3$ be a prime number. Let $(A, +, \circ )$ be the brace constructed in Theorem \ref{main} for fixed $i,k,y\in \mathbb F_{p}$ with $y\neq 0$. Then $(A, +, \circ )$ satisfies the Multiplicative Table above for the same $i,k,y$. 
	\end{theorem}
	\begin{proof} 
		Let $f(g)=\lambda_{g}(1)-1$ be defined as in Theorem \ref{main}. Recall that maps $\lambda _{g}:B'\rightarrow B'$ were defined in Theorem \ref{17}. By Theorem \ref{17} functions $\lambda_{g}$ satisfy relations from (\ref{Table1}) above Theorem \ref{17}. Consider relations given by (\ref{Table1}), each of them is of the form  
		$\lambda_{g}(h')=\alpha_{1}R'+\alpha_{2}Q'+\alpha_{3}P'+\alpha_{4}S'$, for some $g\in \{P, Q, R, S\}$, $h'\in \{P', Q', R', S'\}$ and $\alpha_{1}, \alpha_{2}, \alpha_{3}, \alpha_{4}\in \mathbb F_{p}$. Let $h\in G$ be such that $\lambda _{h}(1)-1=h'$. Notice, that our relation can be written as 
		\begin{align*}
			&(\lambda_{g\circ h}(1) -1)-(\lambda_{g}(1)-1)=\lambda_{g}(\lambda_{h}(1)-1)\\
			&=\alpha_{1}(\lambda_{R}(1)-1)+\alpha_{2}(\lambda_{Q}(1)-1)+\alpha_{3}(\lambda_{P}(1)-1)+\alpha_{4}(\lambda_{S}(1)-1).
		\end{align*} Since $\lambda_{g}(1)-1=f(g)$ we can write it as \[f(g\circ h)-f(g)=\alpha_{1}f(R)+\alpha_{2}f(Q)+\alpha_{3}f(P)+\alpha_{4}f(S).\]So we know that function $f:G\longrightarrow B$ used in Theorem \ref{main} satisfies the above relation. Because $f: G\longrightarrow B$ is a bijective function, we can apply the map $f$ to both sides of this relation to obtain an equivalent relation  \[f^{-1}\left(f(g\circ h)-f(g)\right)=f^{-1}\left(\alpha _{1}f(R)+\alpha _{2}f(Q)+\alpha _{3}f(P)+\alpha _{4}f(S)\right).\]
		Using the formula for addition in brace $(A, +,\circ)$ from Theorem \ref{main} we find that the relation \[g\circ h-g=\alpha _{1}R+\alpha _{2}Q+\alpha _{3}P+\alpha _{4}S\] holds in brace $A$. In this way we can obtain all relations from the Multiplicative Table. 
	\end{proof}
	
	\begin{theorem}\label{100} Let $P, Q, R, S$ be elements of some set $\mathcal{S}$. Let $(A, +, \circ)$ be an $\mathbb F_{p}$-brace such that $A=\mathbb F_{p}R+\mathbb F_{p}Q+\mathbb F_{p}P+\mathbb F_{p}S$ as a linear $\mathbb F_{p}$-space. Suppose that elements $P, Q, R, S$ satisfy the relations from the Multiplicative Table for fixed $i,k,y\in\mathbb F_{p}$ with $y\neq 0$ and $p>3$. Then the following holds.
		\begin{enumerate}
			\item $P, Q, R, S$ satisfy the defining relations of the group XV. So the following holds.
			\[Q\circ S=S\circ Q\circ  P,\ Q\circ R=R\circ Q,\ P\circ S=S\circ P,\ P\circ Q=Q\circ P,\ P\circ R=R\circ P,\]
			\[ R\circ S=S\circ R\circ Q,\ P^{p}=Q^{p}=R^{p}=S^{p}=0.\]
			\item Every element from $\mathbb F_{p}R+\mathbb F_{p}Q+\mathbb F_{p}P+\mathbb F_{p}S$ can be written in an unique way as $R^{\alpha}\circ Q^{\beta}\circ P^{\gamma}\circ S^{\xi}$ for some $\alpha , \beta , \gamma,  \xi \leq p$.  
			\item The multiplication of any two elements from $\mathbb F_{p}R+\mathbb F_{p}Q+\mathbb F_{p}P+\mathbb F_{p}S$ is uniquely determined. Therefore, $A$ is uniquely determined by the Multiplicative Table and by elements $i,k,y$. 
			\item The group $(A, \circ)$ is generated as a group by elements $P, Q, R, S$ and it satisfies the relations of group XV, and it has $p^{4}$ elements, hence it is group XV. 
			\item The brace $(A, +, \circ)$ is not right nilpotent.
		\end{enumerate}
	\end{theorem}
	\begin{proof} 
		\begin{enumerate}
			\item It follows from Theorem \ref{17}, since maps $\lambda_{P}, \lambda_{Q}, \lambda_{R}, \lambda_{S}$ defined as $\lambda_{g}(h)=g\circ h-g$ in brace $A$ have an analogous matrix form as maps $\lambda_{P}, \lambda_{Q}, \lambda_{R}, \lambda_{S}$ from $B'$ to $B'$ in Theorem \ref{17}. It also follows by a direct calculation.
			\item Let $a=\alpha R+\beta Q+\gamma P+\xi S$ for some $\alpha , \beta, \gamma, \xi\in \mathbb F_{p}$. Notice that $Q, P, S\in A^{2}$, by the Multiplicative Table. Consider the element $a_{1}=R^{p-\alpha }\circ a$. It follows that $a_{1}=-\alpha R+a+c$ for $c\in A^{2}$ (by the formula $(d\circ b)*a=d*a+d*(b*a)+b*a$ applied several times). Notice that since $P, Q, S\in A^{2}$ it follows that $R\notin A^{2}$ as otherwise $A=A^{2}$, hence $A=A^{2}=A*A^{2}=A^{3}$, so $A=A^{5}=0$. Consequently, $A^{2}=\mathbb F_{p}Q+\mathbb F_{p}P+\mathbb F_{p}S$. Therefore, \[a_{1}=\beta ' Q+\gamma ' P+\xi ' S,\] for some $\beta ' , \gamma ', \xi '\in \mathbb F_{p}$. 
			
			Consider the element $a_{2}=Q^{p-\beta '}\circ a_{1}$. It follows that $a_{2}\in -\beta 'Q+a_{1}+ A^{3}+A*a_{1}\subseteq A^{3}$ (since $P, S\in A^{3}$ and $A*A^{2}=A^{3}$). Notice that $ P, S\in A^{3}$, hence $Q\notin A^{3}$ as otherwise $A^{2}=A^{3}=A^{4}=A^{5}=0$. Therefore, $a_{2}=\gamma '' P+\xi '' S$ for some $\gamma '', \xi ''\in \mathbb F_{p}.$ By the Multiplicative Table $P*P=P*S=S*P=S*S=0$, so $a_{2}=P^{\gamma ''}\circ S^{\xi ''}$, and consequently $a=R^{\alpha }\circ Q^{\beta '}\circ   P^{\gamma ''}\circ S^{\xi ''}$. This concludes the proof.
			
			\item Let $a, b\in \mathbb F_{p} $. By item 2 above we can write $a=R^{\alpha}\circ Q^{\beta}\circ P^{\gamma}\circ S^{\xi}$ for some $\alpha , \beta , \gamma, \xi \leq p$. Then \[a*b+b=\lambda_{a}(b)=\lambda_{R^{\alpha}}(\lambda_{Q^{\beta }}(\lambda_{P^{\gamma}}(\lambda_{S^{\xi}}(b))),\] and it can be calculated using the Multiplicative Table to give a concrete element from $\mathbb F_{p}R+\mathbb F_{p}Q+\mathbb F_{p}P+\mathbb F_{p}S$. This determines uniquely a brace $(A, +, \circ)$. 
			
			\item Consider the subgroup $A'$ of $(A, \circ )$ generated by elements $R, Q, P, S$. Then by item 2 above we have that $A'$ has cardinality $p^{4}$, so $A'=(A, \circ)$. By item 1 above $A'$ satisfies all the relations of group XV, which concludes the proof (since $A'$ has the same cardinality as group XV).
			
			\item Observe that by Multiplicative Table we obtain $(S*Q)*({\frac {-1}y}R)=S$, so $A$ cannot be right nilpotent.
		\end{enumerate}
	\end{proof}

	\begin{theorem}\label{24} 
		Let $p>3$ be a prime number. Let $(A, +, \circ)$ an $\mathbb F_{p}$-brace which is not right nilpotent with multiplicative group $(A, \circ)$ being the group XV. Then the $\mathbb F_{p}$-brace $A$ is isomorphic to some brace constructed in Theorem \ref{main}. 
	\end{theorem}
	\begin{proof} 
		Let $(A, +, \circ)$ be an $\mathbb F_{p}$-brace which is not right nilpotent with multiplicative group $(A, \circ)$ being the group XV. Then $(A, +, \circ )$  satisfies the relations from the Multiplicative Table by Proposition \ref{listings} for some $i,k,y\in \mathbb F_{p}, y\neq 0$. 
		
		Let $B$ be the brace constructed in Theorem \ref{main} for the same $i,k,y$. Then $B$ satisfies the same Multiplicative Table as $A$, by Theorem \ref{9}. By Theorem \ref{100}, item 3, braces which satisfy the same Multiplicative Table coincide, so $A$ and $B$ are isomorphic braces.
	\end{proof}

	\section{Some Properties of The Constructed Braces}
	Recall that a brace is called prime if product of any two non-zero ideals in this brace is non-zero.
	\begin{proposition} 
		Let $p>3$ be a prime number. Let $A$ be a brace constructed in Theorem \ref{main} for fixed $y, i, k\in \mathbb F_{p}$ with $y\neq 0$. Then  $A$ is a prime brace. Moreover, $A$ contains a non-zero strongly nilpotent (i.e.,  both right and left nilpotent) ideal $A^{2}$.
	\end{proposition}
	\begin{proof} 
		A non-zero ideal $I$ in $A$ contains a non-zero element $c=\alpha_{1}R+\alpha_{2}Q+\alpha_{3}P+\alpha_{4}S$ for some $\alpha_{1}, \alpha_{2}, \alpha_{3}, \alpha_{4}\in \mathbb F_{p}$. If $c\notin A^{4}$, then when multiplying $c$ from the left several times by either $R$ or $Q$ we obtain some non-zero element in $A^{4}$, hence $I$ contains $S\in A^{4}$. Since every ideal containing $S$ contains $Q, P$ we get that $A^{2}\subseteq I$  (since $S*Q=-P$, $S*R=-Q+P-2^{-1}yS$). Notice that $A^{2}*A^{2}\neq 0$ since $Q*Q\neq 0$, so $A$ is a prime brace, as every product of non-zero  ideals in $A$  is non-zero. Moreover, $(A^{2}*A^{2})*A^{2}=A^{2}*(A^{2}*A^{2})=0$, so $A^{2}$ is a strongly nilpotent ideal in $A$.
	\end{proof}
	
	\begin{example}
		Consider a vector space $V$ spanned by the elements $P, Q, R, S$ over the field $\mathbb{F}_p$ with $p$ elements, with multiplication that satisfies \[(a+b)c=ac+bc\ \text{and} \ a(b+c)=ab+ac\] for all $a,b,c$ in $V$. Furthermore, the following relations hold in $V$.
		\begin{align*}
			&PS=SS=QS=RS=PP=SP=QP=PQ=RQ=QR=0,\\
			&RR=jP+kS,\ QQ=-yS,\ RP=yS,\ SQ=-P,\ PR=yS,\ SR=-Q,
		\end{align*}
		where $j,k$, and non-zero $y$ are arbitrary elements from $\mathbb{F}_{p}$. Now, it is easily checked that this defines a pre-Lie algebra by verifying that \[(AB)C-A(BC)=(BA)C-B(AC)\]
		is satisfied for any $A, B,C\in\{P,Q,R,S\}$. Hence $V$ is an example of a pre-Lie algebra which is left nilpotent and not right nilpotent.
		
		The relations in $V$ were obtained using the formula found in \cite{Lazard} for obtaining a pre-Lie algebra from a strongly nilpotent $\mathbb{F}_p$-brace to the brace defined in Theorem \ref{main}, and some relations were guessed as to avoid large calculations.
	\end{example}
	
	\section{Braces whose Multiplicative Group is XIV}\label{S8}
	In this section we will show that all $\mathbb{F}_{p}$-braces whose multiplicative group is XIV are right nilpotent and strongly nilpotent provided that $p$ is a prime number larger than $3$. Recall that $\mathbb F_{p}$ is the field of cardinality $p$.
	\begin{lemma}\label{lem1}
		Let $A$ be an $\mathbb F_{p}$-brace of cardinality $p^{4}$, where $p>3$ is a prime number. Let $i$ be such that $A^{i+1}=0$ and suppose there is $0\neq c\in A^{i}$ such that $c\circ a=a\circ c$ for all $a\in A$. Then $A$ is a right nilpotent brace. 
	\end{lemma}
	\begin{proof}
		Observe that $c*c=0$ hence $c^{j}=jc$ for $j=1,2,\ldots, p $. Denote  $\mathbb F_{p}c=\{c^{j}: j=1,2, \ldots, p\}$ by $I$. Then $I$ is an ideal in brace $A$, and $I*A=A*I=0$, since $I=\{c^{j}: j=1,2, \ldots , p\}$ and $c^{j}$ is a central element in $A$ for every $j$ (recall that  $c^{j}=c\circ \cdots \circ c$). Notice that $A/I$ has cardinality not exceeding $p^{3}$, hence $A/I$ is a right nilpotent brace by \cite{Ba}. Now $A*I=0$ yields that $A$ is a right nilpotent brace.
	\end{proof} 

	\begin{lemma}\label{lem2}
		Let $A$ be an $\mathbb F_{p}$-brace of cardinality $p^{4}$, where $p>3$ is a prime number. Suppose that the multiplicative group $(A, \circ )$ is the group XIV. Denote by $W=\{P^{i}\circ Q^{j}: i,j=1,2, \ldots ,p\}$. Suppose that $A$ is not a right nilpotent brace. Then the following holds.
		\begin{enumerate}
			\item $A^{3}\neq 0$.
			\item Suppose that $Q\in A^{2}$. If $A^{4}=0$ and $A^{3}\neq 0$, then the cardinality of $A^{2}$ is $p^{3}$ and $A^{2}=W+A^{3}$. Moreover $A^{3}*A^{2}=0$.
			\item Suppose that $Q\in A^{2}$. If $A^{5}=0$ and $A^{4}\neq 0$, then the cardinality of $A^{2}$ is $p^{3}$ and $A^{2}=W+A^{4}$. Moreover, $A^{4}*A^{2}=0$.
		\end{enumerate}
	\end{lemma}
	
	\begin{proof} 
		\begin{enumerate}
			\item follows from Lemma \ref{lem1}, since $P=R^{-1}\circ S^{-1}\circ R\circ S\in A^{2}$ by Lemma \ref{2} and $P$ is a central element in $(A, \circ )$. 
			\item Notice that since $Q\in A^{2}$, then $W\subseteq A^{2}$. By Lemma \ref{lem1} we have $W\cap A^{3}=0$, since $W$ consists of central elements. Therefore, $W+A^{3}, A^{3}\circ W\subseteq A^{2}$, and since $A^{2}\neq A$ and $W+A^{3}$ has cardinality $p^{3}$, it follows that $A^{2}$ has cardinality $p^{3}$ and $A^{2}=W+A^{3}= A^{3}\circ W$.
			
			Observe now that \[A^{3}*A^{2}=A^{3}*(W+A^{3})=A^{3}*W+A^{3}*A^{3}=W*A^{3}\subseteq A^{4}=0. \]
			\item Notice that since $Q\in A^{2}$, then $W\subseteq A^{2}$. By Lemma \ref{lem1} we have $W\cap A^{4}=0$, since $W$ consists of central elements. Therefore, $W+A^{4}, A^{4}\circ W\subseteq A^{2}$, and since $A^{2}\neq A$ and $W+A^{4}$ has cardinality $p^{3}$, it follows that $A^{2}$ has cardinality $p^{3}$ and $A^{2}=W+A^{4}=A^{4}\circ W$.
			
			Observe now that \[A^{4}*A^{2}=A^{4}*(W+A^{4})=A^{4}*W+A^{4}*A^{4}=W*A^{4}\subseteq A^{5}=0. \]
		\end{enumerate}
	\end{proof}
	
	\begin{lemma}\label{lem3} 
		Let $A$ be an $\mathbb F_{p}$-brace of cardinality $p^{4}$, where $p>3$ is a prime number. Suppose that the multiplicative group $(A, \circ )$ is the group XIV.  Suppose that $A$ is not a right nilpotent brace. Suppose that $Q\in A^{2}$ and $R\notin A^{2}$. Then $P*R=R*P=0$.
	\end{lemma} 
	\begin{proof}
		Notice that $A/A^{2}$ has cardinality $p$, by Lemma \ref{lem2}. Therefore, $S\circ R^{i}\in A^{2}$ for some $i$, since $R\notin A^{2}$. Let $W$ be as in Lemma \ref{lem2}. By Lemma \ref{lem2} we have $S\circ R^{i}=c\circ w$ for some $w\in W$ and some $c\in A^{i}$ where $A^{i+1}=0$. Therefore $S'=S\circ R^{i}\circ w^{-1}\in A^{i}$ (where $A^{i+1}=0$). Recall that $w$ is a central element, so  $R\circ S'=S'\circ R\circ P$. This implies that \[0=S'*R+S'*P+S'*(R*P)+R*P+P.\] Notice that $S'*A^{2}=0$ by Lemma \ref{lem2}, so $S'*R+R*P+P=0$.
		
		Observe now that $(R\circ S')*R=(S'\circ R\circ P)*R$. Therefore, $R*(S'*R)=R*(P*R)+P*R$, since $S'*A^{2}=0$. Substituting $S'*R=-(R*P+P)$ in the left hand side we get $R*(R*P)+R*P=0$, so $\lambda _{R}(R*P)=0$, hence $R*P=0$, as the  map $\lambda _{a}(b)=a*b+b$ is an invertible map in any brace.
	\end{proof}
	
	\begin{lemma}\label{lem4}  
		Let $A$ be an $\mathbb F_{p}$-brace of cardinality $p^{4}$, where $p>3$ is a prime number. Suppose that the multiplicative group $(A, \circ )$ is the group XIV. Suppose that $A$ is not right nilpotent. Suppose that $Q\in A^{2}$ and $R\in A^{2}$. Then $S\notin A^{2}$ and $P*S=S*P=0$.
	\end{lemma} 
	\begin{proof}
		Notice that  $S\notin A^{2}$ as otherwise $(A, \circ )$ would be generated by elements from $A^{2}$ implying $A=A^{2}$, which is impossible as then \[A=A^{2}=A*A^{2}=A^{3}=A^{4}=A^{5}=0.\]
		
		Notice that $R\in A^{2}$ implies that $R=a\circ w$ for some $a\in A^{i}, w\in W$, where $A^{i+1}=0$, by Lemma \ref{lem2}. Therefore, $R'=R\circ w^{-1}\in A^{i}$. Notice that since $w$ is central in the group $(A, \circ )$, we get $R'\circ S=S\circ R'\circ P$. Consequently, $R'*S=S*P+P$ since $A*R'=0$ and $R'*P\subseteq R'*A^{2}=0$ by Lemma \ref{lem2}.
		
		Observe now that $(R'\circ S)*S=(S\circ R'\circ P)*S$, hence \[R'*(S*S)=S*(R'*S)+S*(R'*(P*S))+S*(P*S)+R'*(P*S)+P*S.\] Recall that $R'*A^{2}=0$ and $A*R'=0$ by Lemma \ref{lem2}; consequently, \[0=S*(R'*S)+P*S+S*(P*S).\] Substituting $R'*S=S*P+P$ in the right hand side we get \[0=S*(S*P+P)+P*S+S*(P*S),\] therefore $2\lambda _{S}(S*P)=0$, so $S*P=0$.
	\end{proof}

	\begin{proposition} 
		Let $A$ be an $\mathbb F_{p}$-brace of cardinality $p^{4}$, where $p>3$ is a prime number. Suppose that the multiplicative group $(A, \circ )$ is the group XIV. Suppose that $A$ is not a right nilpotent brace. Then $Q\notin A^{2}$.
	\end{proposition}
	\begin{proof} 
		Suppose on the contrary that $Q\in A^{2}$. If $R\notin A^{2}$, then $P*R=0$  by Lemma \ref{lem3}. Observe that $A=\mathbb F_{p}R+A^{2}$, since $A/A^{2}$ has cardinality $p$, by Lemma \ref{lem2}. Observe now that for $x,y\in A$ we have $(P\circ x)*y=(x\circ P)*y$, and  by Lemma \ref{2} we get \[P*(x*y)=x*(P*y)\subseteq x*(P*(\mathbb F_{p}R+A^{2}))=x*P*(A^{2})\subseteq x*A^{4}\subseteq A^{5}=0\] (since $P=R^{-1}\circ S^{-1}\circ R\circ S$ is in group $XIV$). Consequently, $P*A^{2}=0 $, and since $P*R=0$, then $P*A=0$. Since $P$ is a central element, we get that $I=\mathbb F_{p}P$ is an ideal in $A$. By \cite{Ba} every brace of cardinality not exceeding $p^{3}$ is right nilpotent, hence $A/I$ is  right nilpotent. Now $A*P=0$ yields that $A$ is right nilpotent.
		
		Suppose now that $R\in A^{2}$. Then by Lemma \ref{lem4} we have that $S\notin A^{2}$, $S*P=P*S=0$.  Observe that $A=\mathbb F_{p}S+A^{2}$, since $A/A^{2}$ has cardinality $p$, by Lemma \ref{lem2}. Observe now that for $x,y\in A$ we have $(P\circ x)*y=(x\circ P)*y$, and by Lemma \ref{2} we get \[P*(x*y)=x*(P*y)\subseteq x*(P*(\mathbb F_{p}S+A^{2}))=x*P*(A^{2})\subseteq x*A^{4}\subseteq A^{5}=0\]  (since $P=R^{-1}\circ S^{-1}\circ R\circ S$ is in group XIV). Consequently, $P*A^{2}=0 $, and since $P*S=0$, then $P*A=0$. Since $P$ is a central element, we get that $I=\mathbb F_{p}P$ is an ideal in $A$. By \cite{Ba} every brace of cardinality not exceeding $p^{2}$ is right nilpotent, hence $A*R$ is  right nilpotent. Now $A*P=0$ yields that $A$ is right nilpotent.
	\end{proof}

	\begin{proposition} 
		Let $A$ be an $\mathbb F_{p}$-brace of cardinality $p^{4}$, where $p>3$ is a prime number. Suppose that the multiplicative group $(A, \circ )$ is the group $XIV$. Suppose that $Q\notin A^{2}$. Then $A$ is a right nilpotent brace.
	\end{proposition}
	\begin{proof}
		Notice that $P\in A^{2}$ by Lemma \ref{2},  since $P=R^{-1}\circ S^{-1}\circ R\circ S$ in the group XIV. By Lemma \ref{lem2} we have that $A/A^{2}$ has cardinality $p$, so $S\circ Q^{i}\in A^{2}, R\circ Q^{j}\in A^{2}$, for some $i,j$, since $Q\notin A^{2}$. We have now the following cases.
		
		{\bf Case 1.} $P\notin A^{3}$, $A^{4}\neq 0$. Notice that this implies that $A^{2}/A^{3}$ has cardinality $p$, and so  there are  $k,l$ such that $S'=S\circ Q^{i}\circ P^{k}\in A^{3}$ and $R'= R\circ Q^{j}\circ P^{l}\in A^{3}$.  Notice that since $Q, P$ are central elements we get
		$P=R'^{-1}\circ S'^{-1}\circ R'\circ S'$, hence by Lemma \ref{2} we have $P\in A^{4}$. By Lemma \ref{lem1} we find that $A$ is right nilpotent. 
		
		{\bf Case 2.} $P\notin A^{3}$, $A^{4}= 0$. Denote by $S'=S\circ Q^{i}\in A^{2}, R'= R\circ Q^{j}\in A^{2}$. We know that $R', S'\in A^{2}$. Notice that since $Q$ is central we get $P=R'^{-1}\circ S'^{-1}\circ R'\circ S'$, hence by Lemma \ref{2} we have $P\in A^{3}$. By Lemma \ref{lem1} we find that $A$ is right nilpotent. 
		
		{\bf Case 3.} $P\in A^{3}$. Notice that $A^{4}\neq 0$ and $P\notin A^{4}$ by Lemma \ref{lem2} since $P$ is central. By Corollary \ref{5} we have $A^{i}/A^{i+1}$ has cardinality $p$ for $i=1,2, 3, 4$. Therefore, $A^{3}=\mathbb F_{p}P+A^{4}$ and $A=\mathbb F_{p}Q+A^{2}$. 
		
		Suppose that there are $R', S'$ such that $R'\circ S'=S'\circ R'\circ P$, and either $S'\in A^{2}, R'\in A^{4}$ or $S'\in A^{2}$ and $R'\in A^{4}$. We will show that $P*Q=0$.
		
		Suppose first that $S'\in A^{4}$ and $R'\in A^{2s}$. Observe that \[(R'\circ S')*Q=(S'\circ R'\circ P)*Q,\] hence $ P*Q=0 $, because $ S'*Q=Q*S'\in A^{5}=0 $, \[A*(P*Q)\subseteq A*(Q*P)\subseteq A^{5}=0,\] \[S'*(R'*Q)=S'*(Q*R')\subseteq S'*A^{3}=S'*(\mathbb F_{p}P+A^{4})=\mathbb F_{p}(S'*P)=0,\] since $S'*P=P*S'\subseteq A^{5}=0$. 
		
		Suppose now that $S'\in A^{2} $ and $ R'\in A^{4}$. Observe that $(R'\circ S')*Q=(S'\circ R'\circ P)*Q$, hence $ P*Q=0 $ since $R'*Q=Q*R'\in A^{5}=0$, \[A*(P*Q)\subseteq A*(Q*P)\subseteq A*A^{4}\subseteq A^{5}=0,\] \[R'*(S'*Q)=R'*(Q*S')\subseteq R'*A^{3}=R'*(\mathbb F_{p}P+A^{4})=\mathbb F_{p}(R'*P)=0,\] since $R'*P=P*R'\subseteq A^{5}=0$. 
		
		We will now show that $P*Q=0$ implies that $A$ is a right nilpotent brace. Observe that $P*A=P*(F_{p}Q+A^{2})=P*A^{2}$, so we need to show that $P*A^{2}=0$. Notice that $(P\circ x)*y=(x\circ P)*y$ since $P$ is central. Consequently, \[P*(x*y)=x*(P*y)=x*(P*(\mathbb F_{p}Q+A^{2}))= \mathbb F_{p}x*(P*Q)+ x*(P*(A^{2}))=0,\] since $P*A^{2}\subseteq A^{4}$ by Corollary \ref{4}, and $P*Q=Q*P\in A^{4}$. 
		
		We will now show that it is possible to find such $R', S'$. We know that $R\circ Q^{j}, S\circ Q^{i} \in A^{2}$. Suppose that $R\circ Q^{j}\in A^{3}$. Then $R\circ Q^{j}\circ P^{k}\in A^{4}$ for some $k$, and we can take $R'=R\circ Q^{j}\circ P^{k}$ and $S'=S\circ Q^{i}$ (since $P, Q$ are central). If $S\circ Q^{i}\in A^{3}$, then $S\circ Q^{i}\circ P^{k}\in A^{4}$ for some $k$, and we can take $S'=S\circ Q^{i}\circ P^{k}$ and $R'=R\circ Q^{j}$.  
		
		If $S\circ Q^{i}, R\circ Q^{j}\notin A^{3}$, then $(S\circ Q^{i})\circ ( R\circ Q^{j})^{t}\in A^{3}$ for some $t$,
		 and then $(S\circ Q^{i})\circ ( R\circ Q^{j})^{t}\circ P^{k}\in A^{4}$ for some $k$. 
		 Therefore, we can take $S'=(S\circ Q^{i})\circ ( R\circ Q^{j})^{t}\circ P^{k}$ and $R'=R\circ Q^{j}$. Observe that $R'\circ S'=S' \circ R'\circ P$, as required. 
	\end{proof}
	
	\begin{corollary}
		Let  $A$ is an $\mathbb F_{p}$-brace for a prime number  $p>3$.  If the  multiplicative group of $A$ is isomorphic to the group XIV, then $A$ is  right nilpotent. 
	\end{corollary}

	\section{Conclusions and Applications}
	In this section we give some motivation for constructing concrete example of braces of a given cardinality. Two main applications of braces are set-theoretic solutions and Hopf-Galois extensions. The connection between braces and Hopf-Galois extensions was first observed by David Bachiller in \cite{DB}. This topic was subsequently investigated in \cite{SVB}, \cite{kayvan} and many other papers by Truman, Kohl, Byott, Childs and others. The following exposition follows the above literature on the topic.
	
	\subsection{Hopf-Galois Extensions Corresponding to a Brace}\label{SS1}
	All Hopf-Galois structures of type $ (A,+) $ arise as regular subgroups of $ \mathrm{Hol}(A,+) $, cf. \cite{DB, Ba, kayvan, kayvan2} for relevant definitions and further details. Now for each brace $ (A,+,\circ) $ we have that $ (A,\circ) $ embeds in $ \mathrm{Hol}(A,+) $ as a regular subgroup through the map 
	\[m: (A,\circ) \longrightarrow \mathrm{Hol}(A,+), \ \ \ a \longmapsto a\lambda_{a}. \] 
	Isomorphic braces correspond to regular subgroups which are conjugate by an element of $ \mathrm{Aut}(A,+) $ inside $ \mathrm{Hol}(A,+) $. On the other hand each regular subgroup $ G\subset \mathrm{Hol}(A,+) $ gives a bijection
	\[\psi : G\longrightarrow (A,+), \ \ \ g\longmapsto g(0), \]
	through which we can define a multiplication on $  (A,+) $ by 
	\[a\circ b = \psi\left(\psi^{-1}(a)\psi^{-1}(b)\right) \]
	which yields a brace $ (A,+,\circ) $.
	
	
	Therefore, given a brace $ (A,+,\circ) $ and an automorphism $ \gamma \in \mathrm{Aut}(A,+) $, we can create, another isomorphic brace $ (A,+,\circ_{\gamma}) $ by embedding 
	\[ 
	\begin{tikzcd}[row sep=2.5em , column sep=2.5em]
		\left(A,\circ\right) \arrow[]{r}{m}  & G \subset  \mathrm{Hol}\left(A,+\right) \arrow[]{d}{C_{\gamma}}[swap]{\wr} & & \\   & G_{\gamma}=\gamma G\gamma^{-1} \arrow[]{r}{\psi} & (A,+). &
	\end{tikzcd} \]
	The new circle operation is through
	\[a\circ_{\gamma} b = \psi\left(\psi^{-1}(a)\psi^{-1}(b)\right) \]
	in terms of holomorph. This corresponds to defining 
	\[	a\circ_{\gamma} b = \gamma^{-1}\left(\gamma(a)\circ\gamma(b)\right) \]
	as in this case $ (A,+,\circ) $ and $ (A,+,\circ_{\gamma}) $ are isomorphic as braces, so they correspond to regular subgroups inside holomorph which are conjugate by $ \gamma $.
	
	Now if $ \gamma \in \mathrm{Aut}(A,+,\circ) $, then $ (A,+,\circ) $ and $ (A,+,\circ_{\gamma}) $ are the same and that means $ G_{\gamma}=G $ inside the holomorph. For a brace $ (A,+,\circ) $ and for each coset representative $ \gamma \in \mathrm{Aut}(A,+)/\mathrm{Aut}(A,+,\circ) $ we find a distinct Hopf-Galois structure $ (A,+,\circ_{\gamma}) $. As we move through the braces with additive group $ (A,+) $ we get all the Hopf-Galois structures of type $ (A,+) $.
	
	Therefore, we can summarise on how to construct Hopf-Galois extensions corresponding to a given brace as follows:
	\begin{enumerate}
		\item Let $(A, +,  \circ)$ be a brace.
		\item Let $\gamma _{i}$ be the automorphisms of the abelian group $(A, +)$.
		\item Check which $\gamma _{i}$ are also automorphisms of  the whole brace $(A,+, \circ)$.
		\item We find coset representatives of $\mathrm{Aut}(A,+)/\mathrm{Aut}(A,+,\circ) $, we call them $\xi _{1}, \xi_{2}, \ldots $. Recall that $ (A,+,\circ_{\alpha}) $ and $ (A,+,\circ_{\beta }) $ are the same in $\mathrm{Aut}(A,+)/\mathrm{Aut}(A,+,\circ) $ if and only if  $\alpha =\beta \gamma $ for some $ \gamma \in \mathrm{Aut}(A,+,\circ) $ (so $\alpha (a)=\beta (\gamma (a))$ for all $a\in A$).
		\item For the automorphisms $\xi _{i}$ of the group $(A, +)$ from our list above we find the brace $(A_{\xi_{i}}, +, \circ_{\xi_{i}})$, with the same set $A_{\xi _{i}}$ as the set $A$ and with the same addition and with the multiplication \[a\circ_{\xi_{i}}b=\xi_{i}^{-1}(\xi_{i}(a)\circ \xi_{i}(b)).\]
		\item For the brace $A_{\xi _{i}}$ the corresponding Hopf-Galois extension is the regular subgroup $(a, \lambda _{a})$ (denoted above as $a\lambda _{a}$) of the holomorph $\mathrm{Hol}(A, +)$ where \[\lambda _{a}(b)=a\circ _{\xi _{i}}b-a.\]
	\end{enumerate}
	Since we now know how to construct Hopf-Galois extensions corresponding to the given brace we can now obtain all Hopf-Galois extensions of type $(A, +)$ with Galois group $G$. Following the beginning of this chapter it can be summarised in the following way.
	\begin{enumerate}
		\item Find all non-isomorphic braces $(B, +, \circ )$ whose additive group $(B, +)$ is isomorphic to the group $(A, +)$ and whose multiplicative group $(B, \circ)$ is isomorphic to $G$. 
		\item We list these non-isomorphic braces as $B_{1}, B_{2},\ldots $. All these braces have the same additive group, which is isomorphic to $(A, +)$,  and the same multiplicative group, which is isomorphic to $G$.
		\item For each brace $B_{i}$ from our list we will construct corresponding Hopf-Galois extensions.
		\item To get all Hopf-Galois extensions of type $(A, +)$ with Galois group $G$ we need to collect all the Hopf-Galois extensions constructed for every brace $B_{i}$. All these Hopf-Galois extensions are distinct.
	\end{enumerate} 
	The above reasoning can be applied to any brace, so in particular can be applied to all braces constructed in our paper. This would involve some calculations to characterise automorphisms of the constructed braces, so we omitted the calculations and only explained the main ideas in this chapter. 

	{\bf Acknowledgements.} The second author is supported by the EPSRC grants EP/V008129/1 and EP/R034826/1.

\end{document}